\documentclass[3p,12pt]{elsarticle}
\usepackage{amsmath,amsfonts,amsthm,amssymb,graphicx}
\usepackage{mathptmx}
\usepackage{makeidx}         
\usepackage{graphicx}        
\newcommand{\ud}{\mathrm{d}}
\usepackage{bm}
\usepackage{amssymb}
\usepackage{amsfonts}
\usepackage{comment}
\usepackage{cases}

\usepackage{xcolor}

\newcommand{\cref}[1]{(\ref{#1})}

\newcommand{\CR}{\mbox{\tiny{CR}}}

\newcommand{\md}{\:\mbox{d}}

\makeindex             

\newtheorem{Def}{Definition}[section]

\newtheorem{remark}[Def]{Remark}
\newtheorem{theorem}{Theorem}[section]

\date{}

\journal{J. of Comp. \& Appl. Math.}

\begin{document}

\begin{frontmatter}



\title{Explicit eigenvalue bounds of differential operators defined by symmetric positive semi-definite bilinear forms}

\author{Xuefeng LIU}

\ead{xfliu.math@gmail.com}

\address{Graduate School of Science and Technology, Niigata University, Japan. }

\begin{abstract}
Recently, the eigenvalue problems formulated with symmetric positive definite bilinear forms have been well investigated with the aim of explicit bounds for the eigenvalues.
In this paper, the existing theorems for bounding eigenvalues are further extended to deal with the case of eigenvalue problems defined by positive semi-definite bilinear forms.
As an application, the eigenvalue estimation theorems are applied to the error constant estimation for polynomial projections.
\end{abstract}


\end{frontmatter}

\section{Introduction}

To give explicit eigenvalue bounds is greatly needed in the field of verified computing for the solution verification for non-linear partial differential equations.
Recently, compared to the classical analysis for qualitative error analysis of eigenvalue approximation, 
the research on explicit bounds of eigenvalues has become a new topic in the field of numerical analysis.

Early work about explicit bound of eigenvalues based on the finite element method (FEM) can be traced back to the work of \cite{Kikuchi-Liu-2007, Liu-Kikuchi-2010, Kobayashi2011,Kobayashi2015}, 
where the upper bounds of various interpolation error constants are considered by estimating the first eigenvalue of the corresponding differential operator.
In \citep{Liu-RIMS-2011, Liu-Oishi-SINUM-2013, Liu-AMC-2015}, the lower bounds for leading eigenvalues of differential operators are provided; see also the work of \cite{CarGed2014,CarGal2014}.

Particularly, the general framework proposed in \cite{Liu-AMC-2015} can be applied to eigenvalue problems formulated as $M(u,v)=\lambda N(u,v)$, where $M,N$ are both symmetric positive definite bilinear forms (see detailed setting of eigenvalue problems in \S 4).
Such a framework has been applied to the eigenvalue problems of the Laplace operator \cite{Liu-AMC-2015}, the Biharmonic operator \cite{liu-you-2018,Liao-Shu-Liu-2019}, the Stokes operator \cite{Xie2-Liu-JJIAM-2018}. 
In \cite{You-Xie-Liu-2019}, the framework is further extended to the case that $N$ is positive semi-definite, and the lower eigenvalue bound for the Steklov eigenvalue problem is provided.
This paper provides a summary of the results of eigenvalue estimation under different settings of $M$ and $N$, and the case that $M$ is positive semi-definite is newly discussed along with a concise method to obtain lower eigenvalue bounds.
As an application, these results are applied to bound the error constants for polynomial projection over 2D and 3D finite elements.

The rest of the paper is organized as follows. 
In \S2, two model eigenvalue problems are introduced. In \S3, the eigenvalue problems are sorted into $3$ cases upon different settings and the theorems to obtain lower eigenvalue bounds are described.
In \S4, the proposed methods in \S3 are applied to the model eigenvalue problems. \S4 displays numerical computation results for the error constant estimation of polynomial projection. 
\S5 is a summary of features of the proposed method in this paper.



\section{Model eigenvalue problems}
We consider the eigenvalue problems formulated with bilinear forms $M,N$ over function space $V$: 
Find $u\in V$ and $\lambda \in R$, such that
\begin{equation}
\label{eq-general-eig-pro}
M(u,v) = \lambda N(u,v)\quad \forall v \in V\:.
\end{equation}
The case that $M(\cdot,\cdot)$ and $N(\cdot,\cdot)$ are symmetric positive-definite is well considered in 
\cite{Liu-AMC-2015}. In this paper, we focus on the problems that either $M$ or $N$ is positive semi-definite.
More detailed setting about $M,N$ and $V$ is given in \S\ref{sec-general-theorems}. Since upper eigenvalue bounds can be easily obtained by using, for example,
 conforming finite element, we only discuss the lower bounds estimation in this paper.


Below, we show two concrete model eigenvalue problems that appear in numerical analysis of finite element methods.
We use the standard notation for Sobolev function spaces; see, e.g., \cite{babuska-osborn-1991}.

\subsection{Model eigenvalue problem 1}

Let us consider the model eigenvalue problem of Laplace operator with homogeneous Neumann boundary condition. \\

Find $u\in H^1(\Omega)$ and $\lambda \ge 0$ such that, 
\begin{equation}
\label{eq:ch2-eig-M-semi-1}
(\nabla u, \nabla v)=\lambda (u,v),\quad \forall v\in H^1(\Omega)\:.
\end{equation}
Let $M(u,v):=(\nabla u, \nabla v)_{\Omega}$. Then $M$ is positive semi-definite over $H^1(\Omega)$. \\

To deal with the non-zero eigenvalue in this model problem,
it is natural to consider the eigenvalue problem over $\mbox{Ker}(M)^{\perp}$ (see definition in (\ref{eq:ker_M_perp})),
%
%
%
%
However, in most cases, it is recommended to avoid solving the eigenvalue problem on the subspace $\mbox{Ker}(M)^{\perp}$, due to the following reasons.
\begin{enumerate}
\item [(1)] 
Since $\mbox{Ker}(M)^{\perp}$ is usually defined through constraint condition, 
one needs more efforts to construct explicit base functions in FEM computation.

\item [(2)] The {\em a priori} error estimation of the project from $\mbox{Ker}(M)^{\perp}$ to FEM space may be complex.
For this model eigenvalue problem, let us take the Crouzeix-Raviart FEM space $V^{\CR}_h$ (see definition in \S 4). 
and define $V^h$ by $V^h =\{ v_h  \in V_h^{\CR} ~|~ \int_\Omega v_h  \md \Omega =0 \}$.
In this case, the Crouzeix-Raviart interpolation operator $\Pi_h^{\CR}$ does not map $\mbox{Ker}(M)^{\perp}$ to $V^h$. 
The projection $P_h:\mbox{Ker}(M)^{\perp} \to V^h$ is no longer a locally defined operator like $\Pi_h^{\CR}$. 
\end{enumerate}




\subsection{Model eigenvalue problem 2: projection error constant}

Let $P_k$ be the projection that maps $u \in L^2(\Omega)$ to the space of polynomials of degree up to $k (k\ge 0)$, with respect to $L^2$ inner product. 
The following estimate is needed in the {\em a priori} error estimation construction for boundary value problems; see, e.g., \cite{Liu-Oishi-SINUM-2013}.
$$
\| u-P_k u\|_{L^2(\Omega)} \le  C  |u|_{H^{1}(\Omega)}\quad  \forall u \in H^{1}(\Omega)\:.
$$
It is easy to see that $\lambda=C^{-2}$ is the first eigenvalue of problem (\ref{eq-general-eig-pro}) with the following settings:
\begin{equation}
\label{eq:model_problem_3_setting}
V:=\{v \in H^1(\Omega):\int_{\partial \Omega} v ds=0\}, ~~~ M(u,v)= (\nabla u, \nabla v), ~~~ N(u,v) = (u-P_k u, v-P_k v)\:.
\end{equation}
Here, $M(\cdot, \cdot)$ is always positive definite; $N(\cdot,\cdot)$ is positive semi-definite for $k\ge 1$.
In case $k=0$, $N$ is positive definite and the eigenvalues here correspond to the positive ones of (\ref{eq:ch2-eig-M-semi-1}).

\section{Eigenvalue problem settings and lower eigenvalue bounds}

\subsection{Eigenvalue problem settings}

Upon the assumptions of $M$ and $N$, let us divide the eigenvalue problems into $3$ cases.

\begin{itemize}

\item Case 1: {Both $M$ and $N$ are positive definite}

\begin{itemize}
\item [($A1$)] Let $\widetilde{V}$ be a Hilbert space, the 
inner  product of which is  $\langle, \rangle$  and the  corresponding norm  denoted by  $\|\cdot \|_{\widetilde{V}}$.
${V}$ and $V^h$ are closed linear subspaces of $\widetilde{V}$, and ${V}^h$ is finite-dimensional.

\item [({$A2$})] $N(\cdot, \cdot)$ is symmetric positive definite bilinear form on $\widetilde{V}$.
The norm $\|\cdot \|_N:=\sqrt{N(\cdot, \cdot)}$ is compact respect to $\|\cdot\|_{\widetilde{V}}$. 
That is, every bounded sequence  of $\widetilde{V}$ under $\| \cdot\|_{\widetilde{V}}$ has a subsequence that is Cauchy under $\|\cdot\|_{N}$.

\item [({$A3$})] $M(\cdot, \cdot)$ is symmetric positive definite bilinear form on $\widetilde{V}$. 
The norm introduced by $M$, denoted by $\|\cdot\|_M:=\sqrt{M(\cdot, \cdot)}$, is equivalent to $\|\cdot\|_{\widetilde{V}}$. \\
%

\end{itemize}

\begin{remark}
The assumption (A1)--(A3) are designed to ease the theoretical analysis. 
For practical problems, the target eigenvalue problems will be configured in the space $V$ and be solved approximately in finite dimensional $V^h$. 
Then, $\widetilde{V}$ is selected as  $\widetilde{V}:=V+V^h$, where $M$ and $N$ should be properly defined.
\end{remark}



\item Case 2: {Positive definite $M$ and positive semi-definite $N$}

To solve the model eigenvalue problem 2 in the previous section, let us replace the (A2) condition for Case 1 with the 
following one. The assumption $(A1)$, $\widetilde{A2}$, $(A3)$, is the essentially the same as the one proposed in \cite{You-Xie-Liu-2019}.

\begin{itemize}

\item [($\widetilde{A2}$)] $N(\cdot, \cdot)$ is symmetric positive semi-definite bilinear form on $\widetilde{V}$.  
The semi-norm $\|\cdot \|_N:=\sqrt{N(\cdot, \cdot)}$ is compact respect to $\|\cdot\|_{\widetilde{V}}$. 

%

\end{itemize}




\item Case 3: {Positive semi-definite $M$ and positive definite $N$}

\label{sec-general-theorems}

In case that $M$ is positive semi-definite, let us first introduce the kernel space of $M$ in $V$ and $V^h$,
\begin{equation}
\label{eq:ker_M_h}
\mbox{Ker}(M):=\{v \in V |  M(v,v) =0 \}, \quad 
\mbox{Ker}_h(M):=\{v_h \in V^h |  M(v_h,v_h) =0 \}\:.
\end{equation}
Also, define the orthogonal complement space of $\mbox{Ker}(M)$ and $\mbox{Ker}_h(M)$
\begin{equation}
\label{eq:ker_M_perp}
\mbox{Ker}(M)^{\perp}:=\{u \in V ~|~  M(u,v) =0, \forall v \in \mbox{Ker}(M)\}\:.
\end{equation}
\begin{equation}
\label{eq:ker_M_h_perp}
\mbox{Ker}_h(M)^{\perp,h}:=\{u_h \in V^h ~|~  M(u_h,v_h) =0, \forall v_h \in \mbox{Ker}_h(M)\}\:.
\end{equation}

To deal with the eigenvalue problem with positive semi-definite $M$, 
let us replace assumption $(A3)$ in Case 1 by ($\widetilde{A3}$) as follows.
\begin{enumerate}



\item [($\widetilde{A3}$)] Let $M$ be a positive semi-definite bilinear form on $\widetilde{V}$ and $\mbox{dim}(\mbox{Ker}(M))<\infty$. 
The norm introduced by  $M$ on $\mbox{Ker}(M)^{\perp}$ is equivalent to $\|\cdot\|_{\widetilde{V}}$. 
Moreover, 
$$
\mbox{Ker}(M)=\mbox{Ker}_h(M)\:.
$$
\end{enumerate}

\end{itemize}

\vskip 0.1cm

For the above $3$ cases, let us defined eigenvalue problems by $M(\cdot,\cdot)$ and $N(\cdot,\cdot)$ over $V$:
Find $u\in V$ and $\lambda \in  \mathbb{R}$, such that,
\begin{equation}
\label{eq:eig-variational-case-1-2}
M(u, v) = \lambda N(u,v)\quad \forall v \in V.
\end{equation}
From arguments of compactness (see, e.g., \S 8 of \cite{babuska-osborn-1991}), the eigenpair of (\ref{eq:eig-variational-case-1-2}) can be denoted by 
$\{\lambda_k, u_k\}$ $(k=1,2, \dots,\infty)$  and
$N(u_i, u_j) =\delta_{ij}$ ($\delta_{ij}$: Kronecker's delta).

It is well known that the eigenvalues in Case 1 and 2 distribute as
$$
0<\lambda_1 \le \lambda_2 \le \cdots\:.
$$
For Case 3, there exist zero eigenvalues such that
$$0\le \lambda_1 \le \lambda_2  \le  \cdots \:.
$$
Moreoever, for any eigenpair $(\lambda_i, u_i)$ with $\lambda_i>0$, we have
$$
\lambda_i N(u_i,v)= M(u_i,v) \le \|u_i\|_M \|v\|_M \quad \forall v \in V.
$$
Thus, $N(u_i, v)=0$ for $v\in \mbox{Ker}(M)$ if $\lambda_i>0$.
\\

Let $n:=\mbox{dim}(V^h)$ for Case 1; $n:=\mbox{dim}(V^h) - \mbox{dim}(\mbox{Ker}_h(N))$ for Case 2, where $\mbox{Ker}_h(N):=\{v_h \in V^h ~ |~ N(v_h,v_h)=0 \}$.
Consider the eigenvalue problem over $V^h$: 
Find $u_h \in V^h$ and $\lambda_h \in \mathbb{R} $, such that,
\begin{equation}
\label{eq:eig-variational-h}
M(u_h, v_h) = \lambda_h N(u_h,v_h)\quad \forall v_h \in V^h.
\end{equation} 
Let $\{(\lambda_{h,k},u_{h,k})\}$ $(k=1,2,\cdots,n)$ be the eigen-pair of (\ref{eq:eig-variational-h}) with
$0 \le \lambda_{h,1} \le \lambda_{h,2} \cdots \le \lambda_{h,n}$. 
The eigenvalues $\lambda_{h,k}$ can be calculated rigorously by solving the corresponding matrix eigenvalue problem $Ax=\lambda Bx$. 
Notice that for each setting of $M$ and $N$ in Case 1, 2 and 3, either matrix $A$ and $B$ is positive definite and both matrices are symmetric.

\subsection{Lower bounds of eigenvalues for Case 1 and 2}


The following Theorem \ref{thm-lower-bound-eig-main} about lower eigenvalue bounds holds for Case 1 and Case 2. 
The proof are provided in \cite{Liu-AMC-2015} and \cite{You-Xie-Liu-2019}, respectively.

\begin{theorem}[Thm. 2.1 of \cite{Liu-AMC-2015}, Thm. 2.4 of \cite{You-Xie-Liu-2019}]
\label{thm-lower-bound-eig-main}
Define $V(h):=V+V^h$.
Let $P_h:V(h) \rightarrow V^h$ be the projection with respect to inner product $M(\cdot, \cdot)$, i.e., for any $u\in V(h)$
\begin{equation}
\label{eq:P_h_definition}
M( u-P_h u, v_h)=0 \quad \forall v_h \in V^h\:.
\end{equation}
Suppose the following error estimation holds for $P_h$: for any $u\in V$,
\begin{equation}
\label{thm-ch}
\| u-P_h u \|_{N} \le C_h \| u-P_h u \|_{M}\:.
\end{equation}
Let $\lambda_k$ and $\lambda_{h,k}$ be the ones defined in (\ref{eq:eig-variational-case-1-2}) and (\ref{eq:eig-variational-h}).
Then, we have
\begin{equation}
\label{thm-lower-bound}
\frac{\lambda_{h,k}}{1 + \lambda_{h,k} C_h^2} \le \lambda_{k} \quad (k=1,2,\cdots, n)\:.
\end{equation}
\end{theorem}

\subsection{Lower bounds of eigenvalues for Case 3}

We prefer to consider the eigenvalue problem defined on the original space $V$, rather than subspace $\mbox{Ker}(M)^{\perp}$. 
For this purpose, let us introduce an interpolation $\Pi_h:V(h)\rightarrow V^h$ satisfying

\begin{itemize}
\item [(1)] Orthogonalty with respect to $M(\cdot,\cdot)$: for any $u \in V$,
\begin{equation}
\label{eq:orthogonality_pi_h}
{M}({u}-\Pi_h {u}, v_h) = 0 \quad \forall v_h\in {V}^h.
\end{equation}
\item [(2)]Invariant of kernel space under $\Pi_h$: 
\begin{equation}
\label{eq:invariant_kernel_space_P_h}
\Pi_h \mbox{Ker}(M)=\mbox{Ker}(M) (=\mbox{Ker}_h(M))\:.
\end{equation}
\end{itemize}
Notice that, $M(\cdot, \cdot)$ is positive semi-definite and thus not an inner product of $V(h)$. However, 
over $\mbox{Ker}(M)^{\perp}$, $M(\cdot, \cdot)$ is positive definite and can be regarded as an inner product.\\

Below is the theorem to provide lower eigenvalue bounds for Case 3.

\begin{theorem}
\label{thm-lower-bound-eig-main-M}
Suppose the following elation holds for $\Pi_h$: for any $u\in V$,
\begin{equation}
\label{thm-ch-M}
\| u- \Pi_h u \|_{N} \le C_h \| u- \Pi_h u \|_{M}\:.
\end{equation}
Let $\lambda_k$ and $\lambda_{h,k}$ be the ones defined in (\ref{eq:eig-variational-case-1-2}) and (\ref{eq:eig-variational-h}).
Then, we have
\begin{equation}
\label{thm-lower-bound-M}
\frac{\lambda_{h,k}}{1 + \lambda_{h,k} C_h^2} \le \lambda_{k} \quad (k=1,2,\cdots, \mbox{dim}(V^h)\:.
\end{equation}
\end{theorem}

\begin{proof}
From the property (\ref{eq:invariant_kernel_space_P_h}), we know $\lambda_k$'s and $\lambda_{h,k}$'s have the same number of zero eigenvalues.
Next, we consider the lower bounds of non-zero eigenvalues, which are based on the result of Theorem  \ref{thm-lower-bound-eig-main}. 
Let $\pi_{0}:V(h) \to  \mbox{Ker}(M)$ be the projection with respect to $N(\cdot, \cdot)$. 
Define $P_h: \mbox{Ker}(M)^{\perp} \to \mbox{Ker}(M)^{\perp,h}$ by 
$$
P_h := (1-\pi_{0}) \Pi_h\:.
$$
Moreover, for $u \in \mbox{Ker}(M)^{\perp}$, 
$$
\pi_{0} \Pi_h u=0,\quad u - P_h u= u-(1-\pi_{0}) \Pi_h u= u - \Pi_h u\:.
$$
Therefore, $P_h$ is a projection in $\mbox{Ker}(M)^{\perp}$ with respect to inner product $M(\cdot, \cdot)$ .
$$
\| u - P_h u\|_N = \| u - \Pi_h u\|_N \le C_h \| u - \Pi_h u\|_M =  C_h \| u - P_h u\|_M\:.
$$
Thus, we can apply Theorem \ref{thm-lower-bound-eig-main} to have the lower bounds for the non-zero eigenvalues.

\end{proof}

\section{Application of Theorem \ref{thm-lower-bound-eig-main} and \ref{thm-lower-bound-eig-main-M} to model eigenvalue problems}

To solve the model eigenvalue problems, let us introduce the Crouzeix-Raviart non-conforming finite element space, which has 
the projection operator reducing to in interpolation operator. Thus, one can give an explicit upper bound for the constant $C_h$ required in (\ref{thm-ch}) and
(\ref{thm-ch-M}).

Let $\Omega \subset \mathbb{R}^d$ ($d=2,3$) be a polygonal domain in 2D space or a polyhedron in 3D space. Suppose $\Omega$ is bounded.
Let $\mathcal{T}_h$ be a face-to-face subdivision of $\Omega$. 
The diameter of element $K \in \mathcal{T}_h$ is denoted by $h_K$ and the mesh size $h$ describes the maximum diameter among all elements $K \in \mathcal{T}_h$. 

The Crouzeix--Raviart finite element space ${V}_h^{\CR}$ is given by
\begin{equation*}
\label{fem_space_CR}
\begin{split}
{V}_h^{\CR} := & \{  v \:|\: v \mbox{ is a piecewise-linear function on $\mathcal{T}_h$;} \int_{e} u \md s \mbox{ is continuous } \\
 & \mspace{2mu} \mbox{across each interior face } e   \}.
\end{split}
\end{equation*}

Since $V_h^{\CR} \not\subset H^1(\Omega)$, we introduce the discrete gradient operator $\nabla_h$, 
which takes the derivatives of $v_h\in V_h^{\CR}$ element-wise. For simplicity, $\nabla_h$ is still written as $\nabla$.
The seminorm $\| \nabla_h v_h \|_{(L^2(\Omega))^2}$ 
is still denoted by $|v_h|_{H^1(\Omega)}$.
Particulary, we can extend the definition of $M(\cdot, \cdot)$ as follows
$$
M(u,v) := \sum_{K\in \mathcal{T}_h} \int_K \nabla_h u \cdot \nabla_h v ~ \mbox{d}K \:.
$$ 

Let us introduce the Crouzeix--Raviart interpolation operator $\Pi_h^{\CR}: H^1(\Omega) \mapsto {V}_h^{\CR}$,
 which is defined element-wise.
For each element $K \in \mathcal{K}_h$, denote the faces by $e_i$ and the nodes by $p_i$, $(i=1,\cdots, d+1)$.
For $u\in H^1(\Omega)$, $(\Pi_h^{\CR} u)|_K$ is a linear polynomial satisfying
\begin{equation}
\int_{e_i} (\Pi_h^{\CR} u)|_K \,\ud s = \int_{e_i}  u \,\ud s, \quad i=1,\cdots, d+1\:.
\end{equation}
%

The interpolation operator $\Pi_{h}^{\CR}$ has the property of orthogonality: for $u \in H^1(\Omega)$,
\begin{equation}
\label{orthogonality_interpolation_CR}
( \nabla_h (\Pi_h^{\CR} u - u), \nabla v_h ) =0 \quad \forall\,   ~ v_h \in V_h^{\CR}\:.
\end{equation}
The following error estimation is provided in \cite{Liu-AMC-2015}. 
\begin{itemize}
\item In case $\Omega \subset \mathbf{R}^2$:
\begin{equation}
\label{cr_consant_2d}
\|u-\Pi_h^{\CR}u\|_{L^2} \le 0.1893h \|_{L^2}\nabla(u-\Pi_h^{\CR}u)\|, \quad u\in H^1(\Omega) \:.
\end{equation}
\item In case $\Omega \subset \mathbf{R}^3$:
\begin{equation}
\label{cr_consant_3d}
\|u-\Pi_h^{\CR}u\|_{L^2} \le 0.3804h \|\nabla(u-\Pi_h^{\CR}u)\|_{L^2}, \quad u\in H^1(\Omega) \:.
\end{equation}

\end{itemize}



The eigenvalue problems in \S2 can be solved with the following spaces settings.

\begin{itemize}
\item Model problem 1

Take $V^h=V_h^{\CR}$. 
It is easy to see that $\Pi_h^{\CR}$ satisfies the conditions (\ref{eq:orthogonality_pi_h}) and
(\ref{eq:invariant_kernel_space_P_h}). Thus, we can apply Theorem \ref{thm-lower-bound-eig-main-M} with 
$\Pi_h := \Pi_h^{\CR}$. The $\pi_0$ operator in the proof of Theorem \ref{thm-lower-bound-eig-main-M} is just the average operator $P_0$.
The constant $C_h$ has explicit bound as shown in (\ref{cr_consant_2d}) and (\ref{cr_consant_3d}). 

%
\item Model problem 2

 Define $V^h:=\{v_h \in V_h^{\CR} | \int_{\partial\Omega} v_h ds=0 \}$.
Then, for any $v \in V$ (see definition of $V$ in (\ref{eq:model_problem_3_setting})), we have $\Pi_h^{\CR} v \in V^h$.
Notice that $\|(I-P_k)u\|_{L^2} \le \|u\|_{L^2}$ for $u\in L_2(\Omega)$. We have
$$
\|(1-\Pi_h^{\CR})u\|_N = \|(I-P_k)(1-\Pi_h^{\CR})u \|_{L^2} \le \|(I-\Pi_h^{\CR})u\|_{L^2} \le C_h \|\nabla (I-\Pi_h^{\CR})\|_{L^2}\:,
$$
where $C_h$ has explicit value in (\ref{cr_consant_2d}) and (\ref{cr_consant_3d}).
Thus, we can apply Theorem \ref{thm-lower-bound-eig-main} along with $P_h:=\Pi_h^{\CR}$.
\begin{remark}
Notice that if $V$ and $V^h$ are selected as 
$$
V:=\{v \in H^1(\Omega):\int_{ \Omega} v d\Omega=0\},\quad
V^h:=\{v_h \in V_h^{\CR} | \int_{\Omega} v_h d\Omega=0 \}\:.
$$
Then, $\Pi_h^{\CR}$ does not map $V$ to $V^h$ any more.

\end{remark}
\end{itemize}

\begin{remark} The two model eigenvalue problems introduced in this paper only involve the first derivative of functions in bilinear forms and the Crouzeix-Raviart FEM space is utilized.
For the eigenvalue problem determined by Biharmonic operators, for example, the error constant estimation for the linear Lagrange interpolation, 
one can turn to the Fujino-Morley FEM space along with the interpolation therein; see applications in \cite{liu-you-2018, Liao-Shu-Liu-2019}. 
\end{remark}

\section{Numerical examples}

In this section, we consider the error estimation for the polynomial projection operator $P_k$ ($k=0,1,2$)
over a triangle and tetrahedron element $T$.
$$
\|u-P_ku\|_{L^2(T)} \le C_k \|\nabla u\|_{L^2(T)},\quad \forall u \in H^1(T)\:.
$$

To provide explicit lower bounds for the constants, one need to solve the corresponding eigenvalue problems by using Theorem \ref{thm-lower-bound-eig-main} and \ref{thm-lower-bound-eig-main-M}: 
$C_k$ ($k\ge 0$) correspond to the first eigenvalue of problem in Case 2, while $C_0$ also corresponds to the second eigenvalue of Case 1.
The lower bounds of constants are obtained by using the quadratic conforming finite element method for each example.

The estimation of $C_k$ below is implemented with interval arithmetic and the eigenvalue problems of matrices are solved by using 
the method of Behnke \cite{Behnke1991} along with interval arithmetic toolbox INTLAB \cite{Ru99a}.

As we will see that error constant $C_1$ is almost but less than half of $C_0$ for either case. Constant $C_2$ has smaller value than $C_1$, but the 
improvement is very limited. This implies that if function $u$ only has $H^1(\Omega)$ regularity, the increased cost in the computation for higher degree $k>1$ 
may not be worth the improvement of projection error.

\subsection{Domain as triangle element}

Let us consider the domain as the following triangles.
$$
K_1:(0,0), (1,0), (0,1);~~ K_2:(0,0), (1,0), (\frac{1}{2},\frac{\sqrt{3}}{2});~~ K_2:(0,0), (\frac{1}{2},0), (\frac{1}{2},\frac{\sqrt{3}}{2})\:. 
$$

The mesh in the computation is created by splitting the triangle domain uniformly for $5$ times.
The estimation of $C_k$ is displayed in Table \ref{table:2d_c_k}.
\begin{table}[h]
\begin{center}
\caption{\label{table:2d_c_k}$C_k$ for triangular domains}

\begin{tabular}{|c|c|c|c|}
\hline
\rule[-3mm]{0mm}{8mm}{}
 & $K_1$ & $K_2$ & $K_3$ \\
\hline
\rule[-3mm]{0mm}{8mm}{}
$C_0$ & $0.318_{3}^{6}$ & $0.238_7^9$ & $0.238_7^9$ \\
\hline
\rule[-3mm]{0mm}{8mm}{}
$C_1$ & $0.17_{55}^{60}$ & $0.13_{78}^{81}$ & $0.13_{17}^{20}$ \\
\hline
\rule[-3mm]{0mm}{8mm}{}
$C_2$ & $0.122_{2}^{8}$ & $0.095_{33}^{63}$ & $0.093_{26}^{58}$ \\
\hline
\end{tabular}
\end{center}
\end{table}

\subsection{Domain as tetrahedron element}

Define vertices $p_i$'s in 3D space as follows,
$$
p_1=(0,0,0),~~ p_2=(1,0,0), ~~p_3=(0,1,0), ~~p_4=(0,0,1), ~~p_5=(1,1,1)\:.
$$
Thus, a cube domain can be divided into tetrahedrons with the shape like $T_{1}=(p_1, p_2, p_3, p_4)$ and  $T_{2}=(p_2, p_3, p_4, p_5)$.
Let $p_6$ and  $p_7$ be the centers of $T_{1}$ and $T_{2}$, respectively.
Let us follow S. Zhang's method to subdivide of a tetrahedron into $4$ sub-tetrahedrons by its centroid, which is needed in stable computation for Stokes equations \cite{Zhang-MC-2014}.
Thus, we have totally $5$ types of sub-tetrahedrons; see Table \ref{table:t_list}.
\begin{table}[h]
\begin{center}
\caption{\label{table:t_list} $5$ types of tetrahedrons}
\begin{tabular}{ccccc}
\hline
$T_{1}$ & $T_{2}$ & $T_3$ & $T_4$ & $T_5$ \\
\hline
$(p_1,p_2,p_3,p_4)$ & $(p_2,p_3,p_4, p_5)$ & $(p_1,p_2,p_3,p_6)$ &  $(p_2,p_3,p_4,p_6)$& $(p_2,p_3,p_4,p_7)$ \\
\hline
\end{tabular}
\end{center}
\end{table}
 
For each $T_i$, we estimate constant $C_k$ ($k=0,1,2$) and display the results in Table \ref{table:3d_c_k}.
The mesh is created by splitting the tetrahedron domain uniformly for $4$ times.
\begin{table}[h]
\begin{center}

\caption{\label{table:3d_c_k}$C_k$ for tetrahedron domains}

\begin{tabular}{|c|c|c|c|c|c|}
\hline
\rule[-3mm]{0mm}{8mm}{}
 & $T_1$ & $T_2$ & $T_3$ & $T_4$ & $T_5$ \\
\hline
\rule[-3mm]{0mm}{8mm}{}
$C_0$ & $0.2_{63}^{70}$ & $0.2_{73}^{84} $ & $0.24_5^9$ & $0.2_{56}^{63}$ & $0.2_{57}^{64}$  \\
\hline
\rule[-3mm]{0mm}{8mm}{}
$C_1$ & $0.1_{62}^{74}$ &  $0.1_{72}^{90}$ & $0.1_{44}^{50} $ & $0.1_{58}^{70}$ & $0.1_{59}^{71}$ \\
\hline
\rule[-3mm]{0mm}{8mm}{}
$C_2$ & $0.1_{12}^{28} $ & $0.1_{19}^{43}$  & $0._{099}^{107}$ & $0.1_{07}^{23}$ & $0.1_{09}^{25}$ \\
\hline
\end{tabular}
\end{center}
\end{table}

\section{Summary}

In this paper, we discuss the eigenvalue problem formulated with bilinear forms $M$, $N$. Particularly, for either $M$ or $N$ being positive semi-definite, 
we show how to obtain explicit lower bounds for the eigenvalue problems along with the non-conforming finite elements.
In future research, we are planning to apply the method proposed here to solve more concrete eigenvalue problems.\\

{\bf Acknowledgement} This research is supported by Japan Society for the Promotion of Science, Grand-in-Aid for Young Scientist (B) 26800090 and Grant-in-Aid for
Scientific Research (C) 18K03411.

\bibliographystyle{plain}
\bibliography{library}

\end{document}